\documentclass[11pt,reqno]{amsart}

\usepackage{amsmath,amssymb,mathrsfs, amsthm}
\usepackage{graphicx,cite,times}
\usepackage{color}

\newtheorem{theo}{Theorem}[section]

\newtheorem{lemm}[theo]{Lemma}

\newtheorem{rema}[theo]{Remark}

\numberwithin{equation}{section}

\def\beq{\begin{equation}}
\def\eeq{\end{equation}}
\def\bea{\begin{eqnarray}}
\def\eea{\end{eqnarray}}
\def\ba{\begin{align}}
\def\ea{\end{align}}

\def\bs{\boldsymbol}

\begin{document}

\title[Inverse moving source problems]{Inverse moving source problems in
electrodynamics}

\author{Guanghui Hu}
\address{Beijing Computational Science Research Center, Beijing 100193, China.}
\email{hu@csrc.ac.cn}

\author{Yavar Kian}
\address{Aix Marseille Univ, Universit\'e de Toulon, CNRS, CPT, Marseille,
France.}
\email{yavar.kian@univ-amu.fr}

\author{Peijun Li}
\address{Department of Mathematics, Purdue University, West Lafayette, Indiana
47907, USA.}
\email{lipeijun@math.purdue.edu}

\author{Yue Zhao}
\address{School of Mathematics and Statistics, Central China Normal University, Wuhan 430079, China}
\email{zhaoyueccnu@163.com}

\maketitle

\begin{abstract}
 This paper is concerned with the uniqueness on two inverse moving source
problems in electrodynamics with partial boundary data. We show that (1) if the
temporal source function is compactly supported, then the spatial source profile
function or the orbit function can be uniquely determined by the tangential
trace of the electric field measured on part of a sphere; (2) if the temporal
function is given by a Dirac distribution, then the impulsive time point and the
source location can be uniquely determined at four receivers on a sphere.
\end{abstract}

\section{Introduction}

Consider the time-dependent Maxwell equation for the electric field $\boldsymbol
E$ in a homogeneous medium
\begin{equation}\label{ef}
 \partial^2_{t}\boldsymbol E(\boldsymbol{x},t) +
\nabla\times(\nabla\times\boldsymbol E
 (\boldsymbol{x},t))=  \boldsymbol F(\boldsymbol x, t),\quad
\boldsymbol{x}\in \mathbb R^3,~ t>0,
\end{equation}
which is supplemented by the homogeneous initial conditions
\begin{align}\label{ic}
\boldsymbol E(\boldsymbol{x},0) = \partial_t \boldsymbol E(\boldsymbol{x},0) =
0, \quad \boldsymbol x \in \mathbb R^3.
\end{align}
We assume that the electrodynamic field is excited by a moving point source
radiating over a finite time period. Specifically, the source function
$\boldsymbol F$ is assumed to be given in the following form
\begin{align*}\label{st}
 \boldsymbol F(\boldsymbol x, t)= \boldsymbol J(\boldsymbol x - \boldsymbol a(t))\,g(t),
\end{align*}
where $\boldsymbol J: \mathbb R^3\rightarrow \mathbb R^3$ is the source profile
function, $g: \mathbb R_+\rightarrow \mathbb R$ the temporal function, and
$\boldsymbol a:\mathbb R_+\rightarrow\mathbb R^3$ is the orbit function of the
moving source. Hence the source term is assumed to be a product of the spatially
moving source function $\boldsymbol J(\boldsymbol x - \boldsymbol a(t))$ and the
temporal function $g(t)$. Physically, the spatially moving source function can
be thought as an approximation of a pulsed signal which is transmitted by a
moving antenna, whereas the temporal function is usually used to model the
evolution of source magnitude in time. Throughout, we make the
following assumptions:
\begin{enumerate}
\item The profile function $\boldsymbol J(\boldsymbol x)$ is compactly supported
in $ B_{\hat R}: = \{\boldsymbol x : |\boldsymbol x| < \hat{R}\}$ for some
$\hat{R}>0$;
\item The source radiates only over a finite time period $[0, T_0]$ for some
$T_0 > 0$, i.e., $g(t) = 0$ for $t\geq T_0$ and $t\leq 0$;
\item The source moves in a bounded domain, i.e., $|\boldsymbol a(t)|< R_1$
for all $t\in \mathbb R_+$ and some $R_1>0$.
\end{enumerate}
These assumptions imply that the source term $\boldsymbol F$ is supported in
$B_R\times(0,T_0)$ for $R>\hat R + R_1$. Unless otherwise stated, we always take
$T:=T_0+\hat R + R_1 +R$ and set
$\Gamma_R:=\{\boldsymbol x\in\mathbb R^3: |\boldsymbol x|=R\}$. Denote by
$\boldsymbol\nu$ the unit normal vector on $\Gamma_R$ and let
$\Gamma\subset\Gamma_R$ be an open subset with a positive Lebesgue measure.

We study the inverse moving source problems of determining the profile
function $\boldsymbol J(\boldsymbol x)$ and the orbit function $\boldsymbol
a(t)$  from boundary measurements of the tangential trace of the electric field
over a finite time interval, $\boldsymbol E(\boldsymbol x, t)\times \boldsymbol
\nu|_{\Gamma \times [0,T]}$. Specifically, we consider the following two
inverse problems:
\begin{enumerate}
\item[(i)] IP1. Assume that $\boldsymbol a(t)$ is known, the inverse problem is
to determine $\boldsymbol J$ from the measurement $\boldsymbol E(\boldsymbol x,
t)\times\boldsymbol\nu, \boldsymbol x\in\Gamma, t\in (0, T)$.

\item[(ii)] IP2. Assume that $\boldsymbol J$ is a known vector function, the
inverse problem is to determine $\boldsymbol a(t), t\in(0,T_0)$ from the
measurement $\boldsymbol E(\boldsymbol x,t)\times\boldsymbol\nu, \boldsymbol
x\in\Gamma, t\in (0, T)$.
\end{enumerate}

The IP1 is a linear inverse problem, whereas the IP2 is a nonlinear inverse
problem. The time-dependent inverse source problems have attracted considerable
attention \cite{ACY-EJAM04, Li2015, RS1989, Ya1998, Li2005, K1992, OPS}.
However, the inverse moving source problems are rarely studied for the wave
propagation.  We refer to \cite{GF15} on the inverse moving source problems by
using the time-reversal method and to \cite{PD89,PD91} for the inverse problems
of moving obstacles. numerical methods can be found in \cite{NIO12-IP,WGLL} to
identify the orbit of a moving acoustic point source. To the best of our
knowledge, the uniqueness result is not available for the inverse moving source
problem, which is the focuse of this paper.

Recently, a Fourier method was proposed for solving inverse source problems for
the time-dependent  Lam\'e system \cite{BHKY} and the Maxwell system
\cite{ZHLL}, where the source term is assumed to be the product of a spatial
function and a temporal function. These work were motivated by the studies on
the uniqueness and increasing stability in recovering compactly supported
source terms with multiple frequency
data \cite{BLLT2015,BLT2010,BLZ2017,LY,LY-JMAA,ZL-AA,ZL-AA}. It is known that
there is no uniqueness for the inverse source problems with a single frequency
data due to the existence of non-radiating sources \cite{AM-IP06}.
In \cite{BHKY,ZHLL}, the idea was to use the Fourier transform and combine with
Huygens' principle to reduce the time-dependent inverse problem into an inverse
problem in the Fourier domain with multi-frequency data. The idea was further
extended in \cite{HuKian} to handle the time-dependent source problems in
elastodynamics where the uniqueness and stability were studied.

In this paper, we use partial boundary measurements of dynamical Dirichlet
data over a finite time interval to recover either the source profile function
or the orbit function. In Sections \ref{sec:3} and \ref{sec:4.2}, we show that
the ideas of \cite{BHKY,ZHLL} and \cite{HuKian} can be used to recover the
source profile function as well as the moving trajectory which lies on a flat
surface. For general moving orbit functions, we apply the moment theory to
deduce the uniqueness under a priori assumptions on the path of the moving
source, see Section \ref{subsec:4.1}. When the compactly supported temporal
function shrinks to a Dirac distribution, we show in Section \ref{sec:5} that
the data measured at four discrete receivers on a sphere is sufficient to
uniquely determine the impulsive time point and to the source location.
This work is a nontrivial extension of the Fourier approach from
recovering the spatial sources to recovering the orbit functions. The latter is
nonlinear and more difficult to handle.

The rest of the paper is organized as follows. In Section \ref{sec:2}, we
present some preliminary results concerning the regularity and well-posedness
of the direct problem. Sections \ref{sec:3} and \ref{sec:4} are devoted to the
uniqueness of IP1 and IP2, respectively. In Section \ref{sec:5}, we show the
uniqueness to recover a Dirac distribution of the source function by using a
finite number of receivers.

\section{The direct problem}\label{sec:2}

In addition to those assumptions given in the previous section, we give some
additional conditions on the source functions:
\[
\boldsymbol J\in H^2(\mathbb R^3), \quad \mbox{div}\,\boldsymbol J=0~ \text{ in
}~ \mathbb R^3,\quad g\in C^1(\mathbb R_+), \quad \boldsymbol a\in C^1(\mathbb
R_+).
\]
It follows from \cite{AM-IP06} that any source function can be decomposed
into a sum of radiating and non-radiating parts. The non-radiating part cannot
be determined and gives rise to the non-uniqueness issue. By the
divergence-free condition of $\boldsymbol J$, we eliminate non-radiating sources
in order to ensure the uniqueness of the inverse problem. Since the source term
$\boldsymbol J$ has a compact support in $B_R\times(0,T)$, we may show the
following result by Huygens' principle.

\begin{lemm}\label{2.1}
It holds that $\boldsymbol E (\boldsymbol x, t) = 0$ for all  $\boldsymbol x \in
B_R, t > T$.
\end{lemm}

The proof of Lemma \ref{2.1} is similar to that of Lemma 2.1 in \cite{ZHLL}.
It states that the electric field $\boldsymbol E$ over $B_R$ must vanish after
time $T$. This property of the electric field plays an important role in the
mathematical justification of the Fourier approach.

Noting $\nabla \cdot \boldsymbol J = 0$, taking the divergence on both sides
of \eqref{ef}, and using the initial conditions \eqref{ic}, we have
\[
\partial^2_t ( \nabla\cdot \boldsymbol E(\boldsymbol x, t) ) =0, \quad
\boldsymbol x \in \mathbb R^3, ~ t>0
\]
and
\[
\nabla \cdot \boldsymbol E(\boldsymbol x, 0) =  \partial_t ( \nabla \cdot
\boldsymbol E(\boldsymbol x, 0))= 0.
\]
Therefore, $\nabla\cdot \boldsymbol E(\boldsymbol x, t) = 0$ for all
$\boldsymbol x\in \mathbb R^3$ and $t>0$.
 In view of the identify $\nabla \times \nabla\times = - \Delta +
\nabla\nabla$, we obtain from \eqref{ef}--\eqref{ic} that
\begin{align}\label{we}
\begin{cases}
 \partial^2_t\boldsymbol E(\boldsymbol{x},t) - \Delta  \boldsymbol
E(\boldsymbol{x},t)=\boldsymbol J(\boldsymbol x - \boldsymbol a(t))g(t), &\quad
\boldsymbol{x}\in \mathbb R^3,~ t>0,\\
\boldsymbol E(\boldsymbol{x},0) = \partial_t \boldsymbol E(\boldsymbol{x},0) =
0, &\quad \boldsymbol{x}\in \mathbb R^3.
\end{cases}
\end{align}

We briefly introduce some notation on functional spaces with the time variable.
Given the Banach space $X$ with norm $||\cdot||_X$, the space $C([0,T];X)$
consists of all continuous functions $f: [0,T]\rightarrow X$ with the norm
\[
||f||_{C([0,T];X)}:=\max_{t\in[0,T]} ||f(t,\cdot)||_X.
\]
The Sobolev space $W^{m,p}(0,T;X),$ where both $m$ and $p$ are positive integers
such that $1\leq m < \infty, 1\leq p < \infty$, comprises all functions
$f\in L^2(0,T;X)$ such that $\partial_t^k f, k=0,1,2,\cdots,m$ exist in the
weak sense and belong to $L^p(0,T;X)$. The norm of $W^{m,p}(0,T;X)$ is given by
\[
||f||_{W^{m,p}(0,T;X)}:=\left(\int_0^T \sum_{k=0}^m ||\partial_t^k f(t,
\cdot)||_X^p\right)^{1/p}.
\]
Denote $H^m = W^{m,2}$.

Now we state the regularity of the solution for the initial value problem
\eqref{we}.  The proof follows similar arguments to the proof of Lemma 2.3 in
\cite{ZHLL} by taking $p=2$.

\begin{lemm}\label{re}
The initial value problem \eqref{we} admits a unique solution $$\boldsymbol E \in
C(0, T; H^{3}(\mathbb R^3))^3 \cap H^\tau(0,T; H^{2-\alpha+1}(\mathbb
R^3))^3,\quad \alpha=1, 2, 3,
$$ which satisfies
\begin{align*}\label{regu}
\|\boldsymbol E\|_{C(0, T; H^{3}(\mathbb R^3))^3} + \|\boldsymbol
E\|_{H^\tau(0,T; H^{2-\tau+1}(\mathbb R^3))^3} \leq C\|g\|_{L^2(0, T)}
\|\boldsymbol J\|_{H^2(\mathbb R^3)^3},
\end{align*}
where $C$ is a positive constant depending on $R$.
\end{lemm}

Applying the Sobolev imbedding theorem, it follows from Lemma \ref{re} that
\[
\boldsymbol E\in  C([0,T];H^2(\mathbb R^3))^3\cap  C^2([0,T];L^2(\mathbb R^3))^3.
\]

Denote by $\mathbb{I}$ the 3-by-3 identity matrix and by $H$ the Heaviside step function.
Recall the Green tensor $\mathbb G(\boldsymbol{x}, t)$ to the Maxwell system
(see e.g., \cite{ZHLL})
\begin{align*}
\mathbb G(\boldsymbol{x}, t) = \frac{1}{4\pi |\boldsymbol x|}
\delta'(|\boldsymbol x| - t) \mathbb I - \nabla \nabla^\top \Big(
\frac{1}{4\pi |\boldsymbol x|} H( |\boldsymbol x|-t) \Big),
\end{align*}
which satisfies
\begin{align*}
&\partial^2_t \mathbb{G}(\boldsymbol{x},t) + \nabla \times( \nabla \times
\mathbb{G}(\boldsymbol{x},t) )= - \delta(t)\delta(\boldsymbol
x)\,\mathbb{I}
\end{align*}
with the homogeneous initial conditions:
\[
 \mathbb{G}(\boldsymbol{x},0) = \partial_t \mathbb{G}(\boldsymbol{x},0)=0, \quad
|\boldsymbol x| \neq 0.
\]
Taking the Fourier transform of $\mathbb G(\boldsymbol{x}, t)$ with respect to
the time variable yields
\begin{align}\label{Grenn's tensor}
\hat{\mathbb G} (\boldsymbol x, \kappa) = \Big( g(\boldsymbol x,
\kappa)\mathbb I+\frac{1}{\kappa^2}\nabla \nabla^\top g(\boldsymbol x, \kappa)
\Big),
\end{align}
which is known as the Green tensor to the reduced time-harmonic Maxwell system
with the wavenumber $\kappa$. Here $g$ is the fundamental solution of the
three-dimensional Helmholtz equation and is given by
\[
 g(\boldsymbol x, \kappa) = \frac{1}{4\pi}\frac{e^{{\rm
i}\kappa|\boldsymbol x|}}{|\boldsymbol x|}.
\]
It is clearly to verify that $\hat{\mathbb G} (\boldsymbol x, \kappa)$
satisfies
\[
\nabla\times(\nabla\times \hat{\mathbb G})-\kappa^2\hat{\mathbb G}=
\delta(\boldsymbol x) \mathbb I, \quad\boldsymbol x\in\mathbb R^3,\;
|\boldsymbol x|\neq 0.
\]

\section{Determination of the source profile function}\label{sec:3}

In this section we consider \text{IP1}. Below we state the
uniqueness result. The idea of the proof is to adopt the Fourier approach of
\cite{ZHLL} to the case of a moving point source.

\begin{theo}\label{us} Suppose that the orbit function $\boldsymbol a$ is given and that
$\int_0^{T_0}g(t){\rm d}t\neq 0$. Then the source profile function $\boldsymbol
J(\boldsymbol x)$ can be uniquely determined by the partial
data set $\{\boldsymbol{E}(\boldsymbol{x},t)\times \boldsymbol\nu: \boldsymbol
x\in \Gamma, t\in (0, T)\}$.
\end{theo}

\begin{proof}
Assume that there are two functions $\boldsymbol J_1$ and $\boldsymbol J_2$
which satisfy
\begin{align*}
\begin{cases}
\partial^2_{t}\boldsymbol E_1(\boldsymbol{x},t)
+ \nabla\times(\nabla\times\boldsymbol E_1 (\boldsymbol{x},t))= \boldsymbol
J_1(\boldsymbol x - \boldsymbol a(t))\,g(t), &\quad
\boldsymbol{x}\in \mathbb R^3,~ t>0,\\
\boldsymbol E_1(\boldsymbol{x},0) = \partial_t \boldsymbol E_1(\boldsymbol{x},0) =
0, &\quad \boldsymbol x \in \mathbb R^3,
\end{cases}
\end{align*}
and
\begin{align*}
\begin{cases}
\partial^2_{t}\boldsymbol E_2(\boldsymbol{x},t)
+ \nabla\times(\nabla\times\boldsymbol E_2 (\boldsymbol{x},t))= \boldsymbol
J_2(\boldsymbol x - \boldsymbol a(t))\,g(t), &\quad
\boldsymbol{x}\in \mathbb R^3,~ t>0,\\
\boldsymbol E_2(\boldsymbol{x},0) = \partial_t \boldsymbol E_2(\boldsymbol{x},0) =
0, &\quad \boldsymbol x \in \mathbb R^3.
\end{cases}
\end{align*}
It suffices to show $\boldsymbol J_1(\boldsymbol x) = \boldsymbol
J_2(\boldsymbol x)$ in $B_R$ if $\boldsymbol{E}_1(\boldsymbol{x},t)\times
\boldsymbol\nu  = \boldsymbol{E}_2(\boldsymbol{x},t)\times \boldsymbol\nu$ for
all $x\in \Gamma, t \in (0, T)$.

Let $\boldsymbol E = \boldsymbol{E}_1 - \boldsymbol{E}_2$ and
\begin{align*}
\boldsymbol f(\boldsymbol x, t)=  \boldsymbol J_1(\boldsymbol x - \boldsymbol
a(t))\,g(t) - \boldsymbol J_2(\boldsymbol x - \boldsymbol a(t))\,g(t).
\end{align*}
Then we have
\begin{align*}
\begin{cases}
\partial^2_{t}\boldsymbol E(\boldsymbol{x},t) +
\nabla\times(\nabla\times\boldsymbol E
 (\boldsymbol{x},t))= \boldsymbol f(\boldsymbol x, t), &\quad
\boldsymbol{x}\in \mathbb R^3,~ t>0,\\
\boldsymbol E(\boldsymbol{x},0) = \partial_t \boldsymbol E(\boldsymbol{x},0) =
0, & \quad \boldsymbol x \in \mathbb R^3, \\
\boldsymbol{E}(\boldsymbol{x}, t) \times \boldsymbol\nu  = 0
,  & \quad \boldsymbol x \in \Gamma, ~ t>0.
\end{cases}
\end{align*}
Denote by $\hat{\boldsymbol E}(\boldsymbol{x}, \kappa)$ the Fourier transform of
$\boldsymbol
E(\boldsymbol x, t)$ with respect to the time $t$, i.e.,
\begin{align}\label{fourier}
\hat{\boldsymbol E}(\boldsymbol{x}, \kappa) = \int_{\mathbb R}
\boldsymbol{E}(\boldsymbol x, t)e^{-{\rm i}\kappa t}{\rm d}t, \quad
\boldsymbol x\in B_R, ~ \kappa \in \mathbb R^{+}.
\end{align}
By Lemma \ref{2.1}, the improper integral on the right hand side of
(\ref{fourier}) makes sense and it holds that
\begin{align*}
\hat{\boldsymbol E}(\boldsymbol{x}, \kappa) = \int_0^T
\boldsymbol{E}(\boldsymbol x, t)e^{-{\rm i}\kappa t}{\rm
d}t<\infty,\quad\boldsymbol x\in B_R,~\kappa>0.
\end{align*}
Hence
\[
\hat{\boldsymbol E}(\boldsymbol{x}, \kappa)\times \boldsymbol\nu
 = 0, \quad\forall \boldsymbol x\in\Gamma, ~\kappa \in
\mathbb R^{+}.
\]
Taking the Fourier transform of \eqref{ef} with respect to the time $t$, we
obtain
\begin{align}\label{eq:1}
\nabla\times(\nabla\times \hat{\boldsymbol E})-\kappa^2\hat{\boldsymbol
E}=\int_0^T\boldsymbol f(\boldsymbol x, t)e^{-{\rm i}\kappa t}{\rm d}t,
\quad\boldsymbol x\in \mathbb R^3.
\end{align}
Since ${\rm supp}(\boldsymbol J) \subset B_{\hat{R}}$ and $|\boldsymbol a(t)|<
R_1$, it is clear to note that $\hat{\boldsymbol E}$ is analytic with respect to
$\boldsymbol x$ in a neighbourhood of $\Gamma_R\supseteq \Gamma$ and
$\hat{\boldsymbol E}$ satisfies the Silver--M\"{u}ller radiation condition:
\begin{equation*}
\lim_{r\to \infty}((\nabla\times\hat{\boldsymbol E}) \times
\boldsymbol x - {\rm i}\kappa r\hat{\boldsymbol E}) = 0, \quad r=|\boldsymbol
x|,
\end{equation*}
for any fixed frequency $\kappa>0$. In fact, the radiation condition of
$\hat{\boldsymbol E}$ can be straightforwardly derived from the expression of
$\boldsymbol E$ in terms of the Green tensor $\mathbb G(\boldsymbol{x}, t)$
together with the radiation condition of $\hat{\mathbb G} (\boldsymbol x\,;
\kappa)$. The details may be found in \cite{ZHLL}. Hence, we have
$\hat{\boldsymbol E}(\boldsymbol{x}, \kappa)\times \boldsymbol\nu
 = 0$ on the whole boundary $\Gamma_R$. It follows from \eqref{Grenn's tensor}
that
\begin{align*}\label{eq:5}
\hat{\boldsymbol E} (\boldsymbol x, \kappa) =  \int_{\mathbb R^3} \hat{\mathbb
G} (\boldsymbol x - \boldsymbol y, \kappa) \int_0^T\boldsymbol f(\boldsymbol y,
t)e^{-{\rm i}\kappa t}{\rm d}t\,{\rm d}\boldsymbol y.
\end{align*}

Let $\hat{\boldsymbol E}\times\boldsymbol\nu$ and $\hat{\boldsymbol
H}\times\boldsymbol\nu$ be the tangential trace of the electric and the magnetic
fields in the frequency domain, respectively. In the Fourier domain, there
exists a capacity operator $T: H^{-1/2}({\rm div}, \Gamma_R)\rightarrow
H^{-1/2}({\rm div}, \Gamma_R)$ such that the following transparent boundary
condition can be imposed on $\Gamma_R$ (see e.g., \cite{N-00}):
\begin{equation}\label{mtbc1}
 \hat{\boldsymbol H}\times\boldsymbol\nu=T(\hat{\boldsymbol
E}\times\boldsymbol\nu)\quad\text{on}~\Gamma_R.
\end{equation}
This implies that $\hat{\boldsymbol H}\times\boldsymbol\nu$ is uniquely determined by $\hat{\boldsymbol E}\times\boldsymbol\nu$ on $\Gamma_R$, provided $\hat{\boldsymbol H}$ and $\hat{\boldsymbol E}$ are radiating solutions.
The transparent boundary condition \eqref{mtbc1} can be equivalently written as
\begin{equation}\label{mtbc2}
(\nabla\times\hat{\boldsymbol E})\times\boldsymbol\nu= {\rm i}\kappa T
(\hat{\boldsymbol E}\times\boldsymbol\nu)\quad\text{on}~\Gamma_R.
\end{equation}

Next we introduce the functions $\hat{\boldsymbol{E}}^{\rm inc}$ and
$\hat{\boldsymbol{H}}^{\rm inc}$ by
\begin{equation}\label{mpw}
\hat{\boldsymbol{E}}^{\rm inc}(x) = \boldsymbol{p}e^{{\rm
i}\kappa\boldsymbol{x} \cdot \boldsymbol{d}}\quad\text{and}\quad
\hat{\boldsymbol{H}}^{\rm inc}(x) = \boldsymbol{q}e^{{\rm
i}\kappa\boldsymbol{x} \cdot \boldsymbol{d}},
\end{equation}
where $\boldsymbol{d} \in \mathbb S^2$ is a unit vector and $\boldsymbol p,
\boldsymbol q$ are two unit polarization vectors satisfying $\boldsymbol
p\cdot\boldsymbol d =0, \boldsymbol q = \boldsymbol p \times \boldsymbol d $. It
is easy to verify that $\hat{\boldsymbol{E}}^{\rm inc}$
and $\hat{\boldsymbol{H}}^{\rm inc}$ satisfy the homogeneous time-harmonic
Maxwell equations in $\mathbb R^3$:
\begin{equation}\label{epw}
\nabla \times (\nabla \times \hat{\boldsymbol{E}}^{\rm inc}) -
\kappa^2 \hat{\boldsymbol{E}}^{\rm inc} = 0
\end{equation}
and
\begin{equation}\label{hpw}
\nabla \times (\nabla \times \hat{\boldsymbol{H}}^{\rm inc}) -
\kappa^2 \hat{\boldsymbol{H}}^{\rm inc} = 0.
\end{equation}

Let $\boldsymbol\xi = -\kappa \boldsymbol d$ with $|\boldsymbol\xi|=\kappa\in
(0, \infty)$. We have from \eqref{mpw} that
$\hat{\boldsymbol{E}}^{\rm inc} =\boldsymbol{p} e^{-{\rm i}\boldsymbol\xi\cdot
\boldsymbol x}$ and $\hat{\boldsymbol{H}}^{\rm inc} = \boldsymbol{q} e^{-{\rm
i}\boldsymbol\xi\cdot \boldsymbol x}$. Multiplying  both sides of \eqref{eq:1}
by $\hat{\boldsymbol{E}}^{\rm inc}$ and using the integration by parts
over $B_R$ and \eqref{epw}, we have from
$\hat{\boldsymbol{E}}(\boldsymbol{x},\kappa) \times \boldsymbol\nu = 0$ on $\Gamma_R$ and the
transparent boundary condition \eqref{mtbc2} that
\begin{align}\label{eq:2}
&\int_{B_R}\int_0^T \boldsymbol f(\boldsymbol x, t) e^{-{\rm i}\kappa t}  \cdot
\hat{\boldsymbol{E}}^{\rm inc} {\rm d}t\;{\rm d}\boldsymbol x \notag\\
= & \int_{B_R} (\nabla\times(\nabla\times \hat{\boldsymbol E})-\kappa^2\hat{\boldsymbol E}) \cdot \hat{\boldsymbol{E}}^{\rm inc}{\rm d}\boldsymbol x\notag\\
=&\int_{\Gamma_R} \nu\times (\nabla\times \hat{\boldsymbol{E}} )\cdot
\hat{\boldsymbol{E}}^{\rm inc}
-\nu\times (\nabla\times \hat{\boldsymbol{E}}^{\rm inc})\cdot
\hat{\boldsymbol{E}} {\rm d}s\notag\\
=& -\int_{\Gamma_R} \left({\rm i}\kappa
T(\hat{\boldsymbol{E}} \times \boldsymbol\nu) \cdot
\hat{\boldsymbol{E}}^{\rm inc}+ (\hat{\boldsymbol{E}} \times
\boldsymbol\nu)\cdot(\nabla \times \hat{\boldsymbol{E}}^{\rm
inc})\right){\rm d}s \notag\\
=& 0.
\end{align}
Hence from \eqref{eq:2} we obtain
\[
\int_{B_R}\int_0^T \boldsymbol{p}e^{- {\rm i}\boldsymbol{\xi} \cdot
\boldsymbol{x}} \cdot g(t)\boldsymbol J_1(\boldsymbol x -\boldsymbol
a(t))e^{-{\rm i}\kappa t}{\rm d}t{\rm d}\boldsymbol x = \int_{B_R}\int_0^T
\boldsymbol{p}e^{- {\rm i}\boldsymbol{\xi} \cdot \boldsymbol{x}} \cdot
g(t)\boldsymbol J_2(\boldsymbol x - \boldsymbol a(t))e^{-{\rm i}\kappa t}{\rm
d}t{\rm d}\boldsymbol x.
\]
By Fubini's theorem, it is easy to obtain
\begin{align}\label{eq:p}
\boldsymbol{p} \cdot \hat{\boldsymbol J}_1(\kappa \boldsymbol d) \int_0^T
g(t)e^{-{\rm i}\kappa \boldsymbol d \cdot \boldsymbol a(t)}e^{-{\rm i}\kappa
t}{\rm d}t = \boldsymbol{p} \cdot \hat{\boldsymbol J}_2(\kappa \boldsymbol d)
\int_0^T g(t)e^{-{\rm i}\kappa \boldsymbol d \cdot \boldsymbol a(t)}e^{-{\rm
i}\kappa t}{\rm d}t.
\end{align}
Taking the limit $\kappa\rightarrow 0^+$ yields
\[
\lim_{\kappa \rightarrow 0}\int_0^T g(t)e^{-{\rm i}\kappa \boldsymbol d \cdot
\boldsymbol a(t)}e^{-{\rm i}\kappa t}{\rm d}t = \int_0^T g(t){\rm d}t > 0.
\]
Hence, there exist a small positive constant $\delta$ such that for all $\kappa \in (0,\delta)$,
\[
\int_0^T g(t)e^{-{\rm i}\kappa \boldsymbol d \cdot \boldsymbol a(t)}e^{-{\rm
i}\kappa t}{\rm d}t\neq 0,
\]
which together with  (\ref{eq:p}) implies that
\begin{align*}
\boldsymbol{p} \cdot \hat{\boldsymbol J}_1(\kappa \boldsymbol d) = \boldsymbol{p} \cdot \hat{\boldsymbol J}_2(\kappa \boldsymbol d).
\end{align*}
Similarly, we may deduce from \eqref{hpw} and the integration by parts that
\begin{align*}
\boldsymbol{q} \cdot \hat{\boldsymbol J}_1(\kappa \boldsymbol d) = \boldsymbol{q} \cdot \hat{\boldsymbol J}_2(\kappa \boldsymbol d) \quad\mbox{for all}\; \boldsymbol d \in \mathbb S^2,\, \kappa \in (0,\delta).
\end{align*}
On the other hand, since $\boldsymbol J_i, ~i=1, 2$ is compactly supported in $B_{\hat{R}}$ and $\nabla_{\boldsymbol x} \cdot \boldsymbol J_i = 0$ in $B_{\hat{R}}$, we  have
\begin{align*}
\int_{\mathbb R^3} \boldsymbol d e^{-{\rm i}\kappa \boldsymbol x \cdot
\boldsymbol d} \cdot \boldsymbol J_i(\boldsymbol x) {\rm d} \boldsymbol x & =-
\frac{1}{{\rm i}\kappa}\int_{B_{\hat{R}}}\nabla e^{-{\rm i}\kappa \boldsymbol x
\cdot \boldsymbol d} \cdot \boldsymbol J_i(\boldsymbol x) {\rm d} \boldsymbol x \\
&= \frac{1}{{\rm i}\kappa}\int_{B_{\hat{R}}} e^{-{\rm i}\kappa \boldsymbol x
\cdot \boldsymbol d} \nabla \cdot \boldsymbol J_i (\boldsymbol x)
{\rm d}\boldsymbol x = 0.
\end{align*}
This implies that $\boldsymbol d \cdot \hat{\boldsymbol J}_i(\kappa \boldsymbol d) = 0.$ Since $\boldsymbol p, \boldsymbol q, \boldsymbol d$ are orthonormal vectors, they form an orthonormal basis in $\mathbb R^3$. It follows from the previous identities that
\begin{align*}
\hat{\boldsymbol J}_1(\kappa \boldsymbol d) &= \boldsymbol{p} \cdot \hat{\boldsymbol{J}}_1(\kappa\boldsymbol d)\boldsymbol{p} +
\boldsymbol{q}\cdot \hat{\boldsymbol{J}}_1(\kappa\boldsymbol d)\boldsymbol{q} +
\boldsymbol{d} \cdot \hat{\boldsymbol{J}}_1(\kappa\boldsymbol d) \boldsymbol{d}\\
& =\boldsymbol{p} \cdot \hat{\boldsymbol{J}}_2(\kappa\boldsymbol d)\boldsymbol{p} +
\boldsymbol{q}\cdot \hat{\boldsymbol{J}}_2(\kappa\boldsymbol d)\boldsymbol{q} +
\boldsymbol{d} \cdot \hat{\boldsymbol{J}}_2(\kappa\boldsymbol d) \boldsymbol{d}\\
& = \hat{\boldsymbol J}_2(\kappa \boldsymbol d)
\end{align*}
for all $\boldsymbol d \in \mathbb S^2$ and $\kappa \in (0,\delta)$.
Noting that $\hat{\boldsymbol J}_i$, $i=1,2$, are analytical functions in
$\mathbb R^3$, we obtain $\hat{\boldsymbol J}_1(\boldsymbol \xi) =
\hat{\boldsymbol J}_2(\boldsymbol \xi)$ for all $\boldsymbol \xi \in \mathbb
R^3$, which completes the proof by taking the inverse Fourier transform.
\end{proof}

\section{ Determination of moving orbit function}\label{sec:4}

In this section, we assume that the source profile function $\boldsymbol J$
is given. To prove the uniqueness for \text{IP2}, we consider two cases:

Case (i): The orbit $\{\boldsymbol a(t): t\in[0,T_0]\}\subset B_{R_1}\cap \mathbb R^3$ is a curve lying in three dimensions;

Case (ii):  $\{\boldsymbol a(t): t\in[0,T_0]\}\subset B_{R_1}\cap\Pi$, where
$\Pi$ is a plane in three dimensions.

The second case means that the path of the moving source lies on a bounded flat surface in three dimensions.
Cases (i) and (ii) will be discussed separately in the subsequent two subsections.

\subsection{Uniqueness to IP2 in case (i)}\label{subsec:4.1}

Before stating the uniqueness result, we need an auxillary lemma.

\begin{lemm}\label{moment}
Let $f_1, f_2, g\in C^1[0, L]$ be functions such that
\[
f^{\prime}_1>0, f^{\prime}_2>0, g>0 \mbox{ on } [0,L];\quad f_1(0)=f_2(0).
\]
In addition, suppose that
\begin{align}\label{eq:f}
\int_{0}^{L} f_1^n( s)g(s){\rm d} s = \int_{0}^{L} f_2^n( s)g(s){\rm d} s
\end{align}
for all integers $n = 0, 1, 2 \cdots$. Then it holds that $f_1 = f_2$ on $[0, L]$.
\end{lemm}

\begin{proof}
Without loss of generality we assume that $f_1(0)=f_2(0)=0$. Otherwise, we may consider the functions
$s\rightarrow f_j(s)-f_j(0)$ in place of $f_j$. To prove Lemma \ref{moment},
we first show $f_1(L)=f_2(L)$ and then apply the moment theory to get $f_1\equiv f_2$.

Assume without loss of generality that $f_1(L)>f_2(L)$. Write $f_1(L) = c$ and
${\rm sup}_{x\in(0,L)}g(x) = M$. Since $f^{\prime}_1(s)>0$ and $f_1(0)=0$, we
have $c>0$. Therefore, there exists sufficiently small positive numbers
$\epsilon > 0$ and $\delta_1, \delta_2>0$ such that
\begin{align*}
f_1(s)\geq c - \delta_1,\; f_2(s)\leq c - 2\delta_1,\; g(s)\geq \delta_2\quad
&\mbox{for all}\quad s \in [L - 2\epsilon, L-\epsilon],\\
f_1(s)>f_2(s)\quad&\mbox{for all}\quad s \in [L - 2\epsilon, L].
\end{align*}
Using the above relations, we deduce from (\ref{eq:f}) that
\begin{align*}
0&=\int_{0}^{L} f_1^n( s)g(s) - f_1^n( s)g(s){\rm d} s\\
&= \int_{L-\epsilon}^{L} f_1^n( s)g(s) - f_2^n( s)g(s){\rm d} s +
\int_{L-2\epsilon}^{L-\epsilon} f_1^n( s)g(s) - f_2^n( s)g(s){\rm d} s \\
&\quad + \int_{0}^{L-2\epsilon}f_1^n( s)g(s) - f_2^n( s)g(s){\rm d} s\\
&\geq \int_{L-2\epsilon}^{L-\epsilon} f_1^n( s)g(s) - f_2^n( s)g(s){\rm d} s -
\int_{0}^{L-2\epsilon}f_2^n( s)g(s){\rm d} s\\
&\geq \epsilon \delta_2\Big[(c - \delta_1)^n - (c - 2\delta_1)^n\Big] - (L -
2\epsilon)M(c - 2\delta_1)^n\\
&\geq (c - \delta_1)^n \Big[\epsilon \delta_2 - (\epsilon \delta_2 + (L -
2\epsilon)M)\Big(\frac{c - 2\delta_1}{c - \delta_1}\Big)^n\Big],
\end{align*}
which means that
\[
(c - \delta_1)^n \Big[\epsilon \delta_2 - (\epsilon \delta_2 + (L -
2\epsilon)M)\Big(\frac{c - 2\delta_1}{c - \delta_1}\Big)^n\Big] \leq 0
\]
for all integers $n = 0, 1, 2 \cdots$. However, since $\frac{c - 2\delta_1}{c -
\delta_1}<1$, there exists a sufficiently large integer $N>0$ such that
\[
\epsilon \delta_2 - (\epsilon \delta_2 + (L - 2\epsilon)M)\Big(\frac{c -
2\delta_1}{c - \delta_1}\Big)^N >0.
\]
Then we obtain
\[
(c - \delta_1)^N \Big[\epsilon \delta_2 - (\epsilon \delta_2 + (L -
2\epsilon)M)\Big(\frac{c - 2\delta_1}{c - \delta_1}\Big)^N\Big] >0,
\]
which is a contradiction. Therefore, we obtain $f_1(L)=f_2(L)$.

Denote $c = f_1(0) = f_2(0)$ and $d = f_1(L) = f_2(L)$.
Since $f_j$ is  monotonically increasing, the relation
$\tau = f_j(s)$ implies that $s=f_j^{-1}(\tau)$ for all $s\in[0, L]$ and
$\tau\in[c,d]$. Using the change of variables, we get
\[
\int_{0}^{L} f_j^n(s)g(s){\rm d} s = \int_{c}^{d}\tau^n g\circ(f_j^{-1}(\tau))
(f_j^{-1})^{\prime}(\tau){\rm d}\tau,\quad j=1,2.
\]
Hence, it follows from (\ref{eq:f}) that
\begin{align}\label{*}
\int_{c}^{d}\tau^n {\rm d}\mu = \int_{c}^{d}\tau^n{\rm d}\nu,
\end{align}
where $\mu$ and $\nu$ are two Lebesgue measures such that
\begin{align*}
{\rm d}\mu = g\circ(f_1^{-1}(\tau)) (f_1^{-1})^{\prime}(\tau){\rm d}\tau,\\
{\rm d}\nu = g\circ(f_2^{-1}(\tau)) (f_2^{-1})^{\prime}(\tau) {\rm d}\tau.
 \end{align*}
 By the Stone--Weierstrass theorem, it is easy to note from \eqref{*} that
${\rm d}\mu = {\rm d}\nu$, which means
\begin{align}\label{ws}
g\circ(f_1^{-1}(\tau)) (f_1^{-1})^{\prime}(\tau) = g\circ(f_2^{-1}(\tau))
(f_2^{-1})^{\prime}(\tau)\quad \text{for all}~\tau \in [c, d].
\end{align}
Introduce two functions
\[
F_1(\tau) = \int^{f_1^{-1}(\tau)}_0 g(s){\rm d}s,\quad F_2(\tau) = \int^{f_2^{-1}(\tau)}_0 g(s){\rm d}s.
\]
Hence, from \eqref{ws} we deduce $F_1^{\prime}(\tau) = F_2^{\prime}(\tau)$ for
$\tau\in [c,d]$. Moreover, since $f_1^{-1}(c) = f_2^{-1}(c)=0$, we have $F_1(c)
= F_2(c)=0$ and then $F_1(\tau) = F_2(\tau)$ for $\tau\in [c,d]$, i.e.,
\begin{align}\label{ws1}
\int^{f_1^{-1}(\tau)}_0 g(s){\rm d}s = \int^{f_2^{-1}(\tau)}_0 g(s){\rm d}s.
\end{align}
From \eqref{ws1}, it is easy to know that $f_1^{-1}(\tau) = f_2^{-1}(\tau)$ for
all $\tau \in [c, d]$. Otherwise, suppose $f_1^{-1}(\tau_0) \neq
f_2^{-1}(\tau_0)$ at some point $\tau_0 \in [c, d]$. Since $g(s)>0$ for all
$s\in(0,L)$, we obtain that
\[
\int^{f_1^{-1}(\tau_0)}_0 g(s){\rm d}s \neq \int^{f_2^{-1}(\tau_0)}_0 g(s){\rm d}s,
\]
which is a contradiction. Consequently, we obtain $f_1^{-1}=f_2^{-1}$
and thus $f_1(s) = f_2(s)$ for all $s\in [0,L]$. The proof is complete.
\end{proof}

Our uniqueness result for the determination of $\boldsymbol a$ is stated as follows.

\begin{theo}\label{ut}
Assume that $\boldsymbol a(0)=\boldsymbol O\in \mathbb R^3$ is located at the
origin and that each component $a_j, j=1,2,3$ of $\boldsymbol a$ satisfies
$|a^{\prime}_i(t)|<1$ for $t \in [0, T_0]$. Then the function $\boldsymbol a(t),
t \in [0, T_0]$ can be uniquely determined by the
data set $\{\boldsymbol{E}(\boldsymbol{x},t)\times \boldsymbol\nu: \boldsymbol
x\in \Gamma, t\in (0, T)\}$.
\end{theo}

\begin{proof}

Assume that there are two orbit functions $\boldsymbol a$ and $\boldsymbol b$ such that
\begin{align*}
\begin{cases}
\partial^2_{t}\boldsymbol E_1(\boldsymbol{x},t)
+ \nabla\times(\nabla\times\boldsymbol E_1 (\boldsymbol{x},t))= \boldsymbol J(x
- \boldsymbol a(t))g(t), &\quad \boldsymbol{x}\in \mathbb R^3,~ t>0,\\
\boldsymbol E_1(\boldsymbol{x},0) = \partial_t \boldsymbol E_1(\boldsymbol{x},0) =
0, & \quad \boldsymbol x \in \mathbb R^3,
\end{cases}
\end{align*}
and
\begin{align*}
\begin{cases}
\partial^2_{t}\boldsymbol E_2(\boldsymbol{x},t) +
\nabla\times(\nabla\times\boldsymbol E_2 (\boldsymbol{x},t))= \boldsymbol J(x -
\boldsymbol b(t))g(t), &\quad \boldsymbol{x}\in \mathbb R^3,~ t>0,\\
\boldsymbol E_2(\boldsymbol{x},0) = \partial_t \boldsymbol E_2(\boldsymbol{x},0) =
0, & \quad \boldsymbol x \in \mathbb R^3.
\end{cases}
\end{align*}
Here we assume that $\bs b(0)=\bs O$ and $|b'_j(t)|<1$ for $t\in[0,t_0]$ and
$j=1,2,3$. We need to show $\boldsymbol a(t) = \boldsymbol b(t)$ in $(0,T_0)$ if
$\boldsymbol{E}_1(\boldsymbol{x},t)\times \boldsymbol\nu (\boldsymbol{x}) =
\boldsymbol{E}_2(\boldsymbol{x},t)\times \boldsymbol\nu$ for $x\in \Gamma, t \in
(0, T)$.

For each unit vector $\boldsymbol{d}$, we can choose two unit polarization vectors
$\boldsymbol p, \boldsymbol q$ such that $ \boldsymbol p\cdot\boldsymbol d=0,
\boldsymbol q = \boldsymbol p \times \boldsymbol d $.
Letting $\boldsymbol E = \boldsymbol E_1 - \boldsymbol E_2$ and following
similar arguments as those of Theorem \ref{us}, we obtain
\begin{align}\label{p}
\boldsymbol p \cdot \hat{\boldsymbol J}(\kappa \boldsymbol d)\int_0^T
g(t)e^{-{\rm i}\kappa \boldsymbol{d}\cdot \boldsymbol a(t)}e^{-{\rm i}\kappa
t}{\rm d}t = \boldsymbol p \cdot \hat{\boldsymbol J}(\kappa \boldsymbol
d)\int_0^T g(t)e^{-{\rm i}\kappa \boldsymbol{d}\cdot \boldsymbol b(t)}e^{-{\rm
i}\kappa t}{\rm d}t,
\end{align}
\begin{align}\label{q}
\boldsymbol q \cdot \hat{\boldsymbol J}(\kappa \boldsymbol d)\int_0^T
g(t)e^{-{\rm i}\kappa \boldsymbol{d}\cdot \boldsymbol a(t)}e^{-{\rm i}\kappa
t}{\rm d}t = \boldsymbol q \cdot \hat{\boldsymbol J}(\kappa \boldsymbol
d)\int_0^T g(t)e^{-{\rm i}\kappa \boldsymbol{d}\cdot \boldsymbol b(t)}e^{-{\rm
i}\kappa t}{\rm d}t,
\end{align}
and
\begin{align*}\label{d}
\boldsymbol d\cdot \hat{\boldsymbol J}(\kappa \boldsymbol d) = 0,
\end{align*}
which means
\[
\hat{\boldsymbol J}(\kappa \boldsymbol d) = \boldsymbol p \cdot \hat{\boldsymbol
J}(\kappa \boldsymbol d) \boldsymbol p + \boldsymbol q \cdot \hat{\boldsymbol
J}(\kappa \boldsymbol d) \boldsymbol q.
\]
Therefore, since $\boldsymbol J\neq 0$, for each unit vector $\boldsymbol d$
there exists a sequence $\{\kappa_j\}_{j=1}^{+\infty}$ such that
$\lim_{j\rightarrow 0}\kappa_j = 0$ and for each $\kappa_j$, either $\boldsymbol
p \cdot \hat{\boldsymbol J}(\kappa_j \boldsymbol d) \neq 0$ or $\boldsymbol q
\cdot \hat{\boldsymbol J}(\kappa_j \boldsymbol d)\neq 0$. Hence from
\eqref{p}--\eqref{q} we have
\begin{align}\label{identity1}
\int_0^T e^{-{\rm i}\kappa_j \boldsymbol{d}\cdot \boldsymbol a(t)}e^{-{\rm
i}\kappa_j t}g(t){\rm d}t = \int_0^T e^{-{\rm i}\kappa_j \boldsymbol{d}\cdot
\boldsymbol b(t)}e^{-{\rm i}\kappa_j t}g(t){\rm d}t, \quad j=1,2,
\cdots.
\end{align}
Expanding $e^{-{\rm i}\kappa_j \boldsymbol{d}\cdot \boldsymbol a(t)}e^{-{\rm
i}\kappa_j t}$ and $e^{-{\rm i}\kappa_j \boldsymbol{d}\cdot \boldsymbol
a(t)}e^{-{\rm i}\kappa_j t}$ into power series with respect to $\kappa_j$,
we write \eqref{identity1} as
\begin{align}\label{ps}
\sum_{n=0}^{\infty}\frac{\alpha_n}{n!}\kappa_j^n = \sum_{n=0}^{\infty}\frac{\beta_n}{n!}\kappa_j^n,
\end{align}
where
\begin{align*}
\alpha_n := \int_0^T (\boldsymbol{d}\cdot \boldsymbol a(t) + t)^n g(t){\rm d}t,
\quad  \beta_n:= \int_0^T (\boldsymbol{d}\cdot \boldsymbol b(t) + t)^n g(t){\rm
d}t, \quad n = 1,2 \cdots.
\end{align*}
In view of the fact that $\mbox{supp}(g)\subset[0, T_0]$, we get
\begin{align*}
\alpha_n = \int_0^{T_0} (\boldsymbol{d}\cdot \boldsymbol a(t) + t)^n g(t){\rm
d}t, \quad  \beta_n= \int_0^{T_0} (\boldsymbol{d}\cdot \boldsymbol b(t) + t)^n
g(t){\rm d}t, \quad n = 1,2 \cdots .
\end{align*}
Since \eqref{ps} holds for all $\kappa_j$ and $\lim_{j\rightarrow
\infty}\kappa_j = 0$, it is easy to conclude that $\alpha_n = \beta_n$ for $n =
0,1,2 \cdots.$ Choosing $\boldsymbol d = (1,0,0)$, we have
\[
( a_1(t) + t)^{\prime} = 1 + a^{\prime}_1(t)>0,\quad ( b_1(t) + t)^{\prime} = 1
+ b^{\prime}_1(t)>0, \quad a_1(0)=b_1(0).
\]
It follows from $\alpha_n = \beta_n$ and Lemma \ref{moment} that $a_1(t) =
b_1(t)$ for $t\in [0, T_0]$. Similarly letting $\boldsymbol d = (0,1,0)$ and
$\boldsymbol d = (0,0,1)$ we have $a_2(t) = b_2(t)$ and $a_3(t) = b_3(t)$ for
$t\in [0, T_0]$, respectively, which proves that $\boldsymbol a(t) = \boldsymbol
b(t)$ for $t \in [0,T_0]$.
\end{proof}

\begin{rema}
In Theorem \ref{ut}, it states that we can only recover the function
$\boldsymbol a(t)$ over the finite time period $[0, T_0]$ because the moving
source radiates in this time period, i.e., $\mbox{supp}(g)=[0,T_0]$. The
information of $\boldsymbol a(t)$ for $t>T_0$ cannot be retrieved. The
monotonicity assumption $a_j'\geq 0$ for $j=1,2,3$ can be replaced by the
following condition: there exist three linearly independent unit directions
$\boldsymbol d_j, j=1,2,3$ such that
\[
|\boldsymbol d_j\cdot \boldsymbol a' (t)|<1, \quad t\in[0,T_0],\;j=1,2,3.
\]
Note that this condition can always be fulfilled if the source moves along a
straight line with the speed less than one.
\end{rema}

\subsection{Uniqueness to IP2 in case (ii)}\label{sec:4.2}

For simplicity of notation, let $ \tilde{\boldsymbol x} = (x_1, x_2, 0)$
for $\boldsymbol x=(x_1,x_2,x_3)$ and $\mathbb R^2 = \{\boldsymbol
x \in \mathbb R^3:  x_3 = 0\}$.
Let $\tilde{\bs a}(t)\in \mathbb R^2$ for all $t\in[0,T_0]$.
In this subsection, we assume that
\[
\bs F(\boldsymbol x, t)=\bs J (\tilde{\bs x}-\tilde{\bs
a}(t))\,h(x_3)\,g(t),\quad \bs x\in \mathbb R^3,\, t\in \mathbb R_+,
\]
where $\bs J(\bs x)=(J_1(\tilde{\bs x}), J_2(\tilde{\bs x}),0)\in H^2 (\mathbb
R^2)^3$ depends only on $\tilde{\bs x}$ and
$h \in H^2(\mathbb R)$, ${\rm supp} (h) \subset (-\hat{R}, \hat{R})\sqrt{2}/2$.
Moreover, we assume that $h$ does not vanish identically and
\[
\mbox{supp}(\bs J)\subset\{\tilde{\bs x}\in \mathbb R^2: |\tilde{\bs
x}|<\hat{R}\sqrt{2}/2\},\quad \nabla_{\tilde{\bs x}} \bs J(\tilde{\bs x})=0.
\]
The temporal function $g$ is defined the same as in the previous sections. The
above assumptions imply that we still have $\mbox{supp}(\bs F)\subset
B_{\hat{R}}\times[0,T_0]$ and $\mbox{div}\,\bs F=0$ in $\mathbb R^3$.
We consider the inhomogeneous Maxwell system
\begin{align}\label{mep3}
\begin{cases}
\partial^2_{t}\boldsymbol E(\boldsymbol{x},t) +
\nabla\times(\nabla\times\boldsymbol E  (\boldsymbol{x},t)) = \boldsymbol
J(\tilde{\boldsymbol x} - \tilde{\boldsymbol
a}(t))\,h(x_3)\,g(t), &\quad\boldsymbol x\in\mathbb R^3,\, t>0,\\
\boldsymbol E(\boldsymbol{x},0) = \partial_t \boldsymbol E(\boldsymbol{x},0) =
0, &\quad \boldsymbol x \in \mathbb R^3.
\end{cases}
\end{align}
Since the equation \eqref{mep3} is a special case of \eqref{ef}, the
results of Lemmas \ref{2.1} and \ref{re} also apply to this case.

For our inverse problem, it is assumed that $\bs J\in \mathcal{A}$ is a given
source function, where the admissible set
\begin{align*}
\mathcal{A} = \{\boldsymbol J=(J_1, J_2, 0): J_i(0)>J_i(\tilde{\boldsymbol
x})~\text{~for~}~ i=1 ~\text{or} ~i= 2 ~\text{and~all}~\tilde{\boldsymbol x}\neq
0 \}.
\end{align*}
The $x_3$-dependent function $h$ is also assumed to be given. We point out that
these a priori information of $\bs J$ and $h$ are physically reasonable, while
$\bs J$ and $h$ can be regarded as approximation of the Dirac functions (for
example, Gaussian functions) with respect to $\tilde{\bs x}$ and $x_3$,
respectively. Our aim is to recover the unknown orbit function
$\tilde{\boldsymbol a}(t)\in C^1([0, T_0])^2$ which has a upper bound
$|\tilde{\boldsymbol a}(t)|\leq R_1$ for some $R_1>0$ and for all $t\in[0,
T_0]$. Let $R > \hat{R} + R_1$ and $T = T_0 + R
+ \hat{R} + R_1.$

Below we prove that the tangential trace of the dynamical magnetic field on
$\Gamma_R\times(0,T)$ can be uniquely determined by that of the electric field.
It will be used in the subsequent uniqueness proof with the data measured on the
whole surface $\Gamma_R$.

\begin{lemm}\label{dtn}
Assume that the electric field $\boldsymbol E\in  C([0,T];H^2(\mathbb
R^3))^3\cap  C^2([0,T];L^2(\mathbb R^3))^3$ satisfies
\[
\begin{cases}
\partial^2_{t}\boldsymbol E(\boldsymbol{x},t) +
\nabla\times(\nabla\times\boldsymbol E  (\boldsymbol{x},t))=  0, &\quad
|\boldsymbol x|>R,~ t\in (0,T),\\
\boldsymbol E(\boldsymbol{x},0) = \partial_t \boldsymbol E(\boldsymbol{x},0) =
0, &\quad \boldsymbol x \in \mathbb R^3.
\end{cases}
\]
If $\boldsymbol E\times \boldsymbol \nu = 0$ on $\Gamma_R \times (0,T)$, then
$(\nabla \times \boldsymbol E)\times \boldsymbol \nu = 0$ on $\Gamma_R \times
(0,T)$.
\end{lemm}

\begin{proof}
Let us assume that $\boldsymbol E\times \boldsymbol \nu = 0$ on $\Gamma_R \times
(0,T)$ and consider $\boldsymbol V$ defined by $$ \boldsymbol
V(\boldsymbol{x},t) =\int_0^t\boldsymbol E(\boldsymbol{x},s){\rm d}s,\quad
(\boldsymbol x,t)\in \mathbb R^3 \times (0,T).$$
In view of \eqref{dtn} and the fact that $\boldsymbol E(\boldsymbol x, t)\times
\boldsymbol \nu = 0$ on $\Gamma_R \times (0,T)$, we find
\begin{align}\label{dtn2}
\begin{cases}
\partial^2_{t}\boldsymbol V(\boldsymbol{x},t) +
\nabla\times(\nabla\times\boldsymbol V
 (\boldsymbol{x},t))=  0, &\quad
|\boldsymbol x|>R,~ t\in (0,T),\\
\boldsymbol V(\boldsymbol{x},0) = \partial_t \boldsymbol V(\boldsymbol{x},0) =
0, &\quad \boldsymbol x \in \mathbb R^3,\\
\partial_t\boldsymbol V(\boldsymbol{x},t)\times \boldsymbol \nu(x)=0, &\quad
(\boldsymbol x,t)\in \Gamma_R \times (0,T).
\end{cases}
\end{align}
We define the energy $\mathcal E$ associated to $\boldsymbol V$ on $\Omega:=
\{x\in\mathbb R^3:\ |x|>R\}$
\[
\mathcal E(t):=\int_\Omega (|\partial_t\boldsymbol
V(\boldsymbol{x},t)|^2+|\nabla_x\times \boldsymbol V(\boldsymbol{x},t)|^2){\rm
d}\bs x,\quad t\in[0,T].
\]
Since $\boldsymbol E\in  C([0,T];H^2(\mathbb R^3))^3\cap   C^2([0,T];L^2(\mathbb
R^3))^3$, we have $$\boldsymbol V\in  C([0,T];H^2(\mathbb R^3))^3\cap
C^1([0,T];H^1(\mathbb R^3))^3\cap C^2([0,T];L^2(\mathbb R^3))^3.$$
 It follows that $\mathcal E\in C^1([0,T])$. Moreover, we get
$$\mathcal E'(t)=2\int_\Omega [\partial_t^2\boldsymbol V(x,t)\cdot \partial_t\boldsymbol V(\boldsymbol x,t) +(\nabla_x\times \boldsymbol V(\boldsymbol x,t))\cdot(\nabla_{\boldsymbol x}\times \partial_t\boldsymbol V(\boldsymbol x,t))]\, {\rm d}\boldsymbol x.$$
Integrating by parts in $\bs x\in\Omega$ and applying \eqref{dtn2}, we obtain
\begin{align*}
\mathcal E'(t)&=2\int_\Omega [\partial_t^2\boldsymbol
V+\nabla_{\boldsymbol x}\times(\nabla_{\boldsymbol x}\times \boldsymbol V)]
\cdot\partial_t\boldsymbol V(\boldsymbol x,t)\,{\rm d}\boldsymbol x\\
&\quad +2\int_{\Gamma_R}(\nabla_{\boldsymbol x}\times \boldsymbol
V)\cdot(\nu\times \partial_t\boldsymbol V(\boldsymbol x,t)){\rm d}s\\
&=0.
\end{align*}
This proves that $\mathcal E$ is a constant function. Since
\[
\mathcal E(0)=\int_\Omega (|\partial_t\boldsymbol
V(\boldsymbol{x},0)|^2+|\nabla_x\times \boldsymbol V(\boldsymbol{x},0)|^2){\rm
d}\bs x=0,
\]
we deduce $\mathcal E(t)=0$ for all $t\in[0,T]$. In particular, we have
\[
\int_\Omega |\boldsymbol E(\boldsymbol{x},t)|^2\,{\rm d}\bs x=\int_\Omega
|\partial_t\boldsymbol V(\boldsymbol{x},t)|^2\,{\rm d}\bs x\leq \mathcal
E(t)=0,\quad t\in[0,T].
\]
This proves that
\[
\boldsymbol E(\boldsymbol{x},t)=0, \quad |\boldsymbol x|>R,~ t\in (0,T),
\]
which implies that $(\nabla \times \boldsymbol E)\times \boldsymbol \nu = 0$ on
$\Gamma_R \times (0,T)$ and completes the proof.
\end{proof}

In the following lemma, we present a uniqueness result for recovering
$\tilde{\bs a}$ from the tangential trace of the electric field measured on
$\Gamma_R$. Our arguments are inspired by a recent uniqueness result
\cite{HuKian} to inverse source problems in elastodynamics. Compared to the
uniqueness result of Theorem \ref{ut}, the slow moving assumption of the source
is not required in the following Theorem \ref{up}.

\begin{theo}\label{up}
Assume that $\boldsymbol J \in \mathcal{A}$ and the non-vanishing function $h$
are both known. Then the function $\tilde{\boldsymbol a}(t), t \in [0, T_0]$
can be uniquely determined by the data set
$\{\boldsymbol{E}(\boldsymbol{x},t)\times \boldsymbol\nu: \boldsymbol
x\in \Gamma_R, t\in (0, T)\}$.
\end{theo}

\begin{proof}
Assume that there are two functions $\tilde{\boldsymbol a}$ and
$\tilde{\boldsymbol b}$ such that
\begin{align}\label{mep1}
\begin{cases}
\partial^2_{t}\boldsymbol E_1(\boldsymbol{x},t) +
\nabla\times(\nabla\times\boldsymbol E_1
 (\boldsymbol{x},t))= \boldsymbol J(\tilde{\bs x} - \tilde{\boldsymbol
a}(t))h(x_3)g(t), &\quad
\boldsymbol{x}\in \mathbb R^3,~ t>0,\\
\boldsymbol E_1(\boldsymbol{x},0) = \partial_t \boldsymbol E_1(\boldsymbol{x},0) =
0, &\quad \boldsymbol x \in \mathbb R^3,
\end{cases}
\end{align}
and
\begin{align}\label{mep2}
\begin{cases}
\partial^2_{t}\boldsymbol E_2(\boldsymbol{x},t) +
\nabla\times(\nabla\times\boldsymbol E_2
 (\boldsymbol{x},t))= \boldsymbol J(\tilde{\bs x} - \tilde{\boldsymbol
b}(t))h(x_3)g(t), &\quad
\boldsymbol{x}\in \mathbb R^3,~ t>0,\\
\boldsymbol E_2(\boldsymbol{x},0) = \partial_t \boldsymbol E_2(\boldsymbol{x},0) =
0, &\quad \boldsymbol x \in \mathbb R^3.
\end{cases}
\end{align}
It suffices to show that $\tilde{\boldsymbol a}(t) = \tilde{\boldsymbol b}(t)$
in $(0,T_0)$ if $\boldsymbol{E}_1(\boldsymbol{x},t)\times \boldsymbol\nu  =
\boldsymbol{E}_2(\boldsymbol{x},t)\times \boldsymbol\nu$ for $x\in \Gamma_R, t
\in (0, T)$.
Denote $\boldsymbol E = \boldsymbol E_1 - \boldsymbol E_2$ and
\begin{align*}
\boldsymbol f(\tilde{\bs x}, t) &= \boldsymbol J(\tilde{\bs x} -
\tilde{\boldsymbol a}(t))g(t) - \boldsymbol J(\tilde{\bs x} - \tilde{\boldsymbol
b}(t))g(t).
\end{align*}
Subtracting \eqref{mep1} from \eqref{mep2} yields
\begin{align}
\begin{cases}
\partial^2_{t}\boldsymbol E(\boldsymbol{x},t) +
\nabla\times(\nabla\times\boldsymbol E
 (\boldsymbol{x},t))= \boldsymbol f(\tilde{\bs x}, t)h(x_3)g(t), &\quad
\boldsymbol{x}\in \mathbb R^3,~ t>0,\\
\boldsymbol E(\boldsymbol{x},0) = \partial_t \boldsymbol E(\boldsymbol{x},0) =
0, &\quad \boldsymbol x \in \mathbb R^3.
\end{cases}
\end{align}
Since $h$ does not vanish identically, we can always find an interval $\Lambda=(a_-,a_+)\subset \mathbb R_+$ such that
\begin{equation}\label{h}
\int_{-\hat{R}\sqrt{2}/2}^{\hat{R}\sqrt{2}/2} e^{\lambda x_3}h(x_3){\rm
d}x_3\neq 0,\quad\forall \lambda\in \Lambda.
\end{equation}
Set $H:=\{(x_1, x_2): a_-^2< x_2^2-x_1^2< a_+^2, x_1>0, x_2>0\}$, which is an open set in $\mathbb R^2$.
We choose a test function $\boldsymbol F(\boldsymbol x, t)$ of the form
\[
\boldsymbol F(\boldsymbol x, t) = \tilde{\bs p}e^{-{\rm i}\kappa_1 t}e^{-{\rm
i}\kappa_2 \tilde{\bs d}\cdot \tilde{\bs x}}e^{\sqrt{\kappa_2^2 -
\kappa_1^2}x_3},
\]
where $\tilde{\bs d}= (d_1, d_2,0)$ is a unit vector, $\tilde{\bs p}=(p_1, p_2,
0)$ is a unit vector orthogonal to $\tilde{\bs d}$, and $\kappa_1, \kappa_2$
are positive constants such that $(\kappa_1, \kappa_2)\in H$. It is easy to
verify that
\begin{align}\label{F}
\partial^2_{t}\boldsymbol F(\boldsymbol{x},t) +
\nabla\times(\nabla\times\boldsymbol F
 (\boldsymbol{x},t)) = 0.
\end{align}
Since $\boldsymbol E(\boldsymbol x, t) \times \boldsymbol \nu = 0$ on
$\Gamma_R$, from Lemma \ref{dtn}, we also have $(\nabla \times \boldsymbol
E(\boldsymbol x, t))\times \boldsymbol \nu = 0$ on $\Gamma_R$. Consequently,
multiplying both sides of the Maxwell system by $\boldsymbol F$ and using
integration by parts over $[0, T] \times B_R,$ we can obtain from \eqref{F}
that
\begin{align*}
& \int_0^{T}\int_{B_R}\boldsymbol f(\tilde{\bs x},t)h(x_3)\cdot \boldsymbol
F(\boldsymbol x, t){\rm d}{\boldsymbol x}{\rm d}t\\
&=\int_0^{T}\int_{B_R} \Big(\partial^2_{t}\boldsymbol E(\boldsymbol{x},t) +
\nabla\times(\nabla\times\boldsymbol E (\boldsymbol{x},t))\Big) \cdot
\boldsymbol F(\boldsymbol x, t){\rm d}{\boldsymbol x}{\rm d}t\\
&=\int_0^{T}\int_{\Gamma_R} \boldsymbol \nu \times (\nabla \times \boldsymbol
E(\boldsymbol x, t)) \cdot \boldsymbol F(\boldsymbol x, t) - \boldsymbol \nu
\times (\nabla \times \boldsymbol F(\boldsymbol x, t)) \cdot \boldsymbol
E(\boldsymbol x, t){\rm d}s{\rm d}t\\
&=\int_0^{T}\int_{\Gamma_R} \boldsymbol \nu \times (\nabla \times \boldsymbol
E(\boldsymbol x, t)) \cdot \boldsymbol F(\boldsymbol x, t) - (\boldsymbol
E(\boldsymbol x, t) \times \boldsymbol \nu) \cdot (\nabla \times \boldsymbol
F(\boldsymbol x, t)) {\rm d}s{\rm d}t\\
& = 0.
\end{align*}
Note that in the last step we have used Lemma \ref{dtn}. Recalling the
definition of $\bs F$ and $\bs f$,
we obtain from the previous identity that
\begin{align*}
\Big(\int_{-\hat{R}\sqrt{2}/2}^{\hat{R}\sqrt{2}/2} e^{\sqrt{\kappa_2^2 -
\kappa_1^2}x_3}h(x_3){\rm d}x_3\Big)~ \tilde{\bs p} \cdot
\int_0^{T}\int_{B_{\hat{R}}}\boldsymbol f(\tilde{\bs x}, t)e^{-{\rm i}\kappa_1
t}e^{-{\rm i}\kappa_2 \tilde{\bs d}\cdot \tilde{\bs x}}{\rm d}\tilde{\bs x}{\rm
d}t = 0.
\end{align*}
In view of (\ref{h}) and the choice of $\kappa_1, \kappa_2$, we get
\begin{align*}
\tilde{\bs p} \cdot \int_0^{T}\int_{B_{\hat{R}}}\boldsymbol f(\tilde{\bs x},
t)e^{-{\rm i}\kappa_1 t}e^{-{\rm i}\kappa_2 \tilde{\bs d}\cdot \tilde{\bs
x}}{\rm d}\tilde{\bs x}{\rm d}t =0.
\end{align*}
For a vector $\boldsymbol v(\tilde{\bs x}, t) \in \mathbb R^3$, denote by
$\hat{\boldsymbol v}(\boldsymbol \xi), \boldsymbol \xi \in \mathbb R^3$ the
Fourier transform of $\boldsymbol v$ with respect to the variable $(\tilde{\bs
x}, t)$, i.e.,
\[
\hat{\boldsymbol v}(\boldsymbol \xi) = \int_{\mathbb R^3}\boldsymbol
v(\tilde{\bs x}, t) e^{-{\rm i}\boldsymbol \xi \cdot (\tilde{\bs x}, t)}{\rm
d}\tilde{\bs x}{\rm d}t.
\]
Consequently, it holds that
\begin{align*}
\tilde{\bs p} \cdot \hat{\boldsymbol f}( \kappa_2\tilde{\bs d},\kappa_1) = 0
\end{align*}
for all $\kappa_2 > \kappa_1 >0$ and $|\tilde{\boldsymbol d}| = 1.$

On the other hand, since $\nabla_{\tilde{\bs x}}\cdot \boldsymbol J = 0,$ we have $\nabla_{\tilde{\bs x}}\cdot \bs f = 0$. Hence,
\begin{align*}
&\quad\tilde{\bs d}\cdot \int_0^{T}\int_{B_{\hat{R}}}\bs f(\tilde{\bs x}, t)e^{-{\rm i}\kappa_1 t}e^{-{\rm i}\kappa_2 \tilde{\bs d}\cdot \tilde{\bs x}}{\rm d}\tilde{\bs x}{\rm d}t\\
&= -\frac{1}{{\rm i}\kappa_2}\int_0^{T}\int_{B_{\hat{R}}}\bs f(\tilde{\bs x}, t) \cdot \nabla_{\tilde{\bs x}} e^{-{\rm i}\kappa_2 \tilde{\bs d}\cdot \tilde{\bs x}}{\rm d}\tilde{\bs x}{\rm d}t\\
& = \frac{1}{{\rm i}\kappa_2}\int_0^{T}\int_{B_{\hat{R}}} \nabla_{\tilde{\bs x}}\cdot \bs f(\tilde{\bs x}, t)e^{-{\rm i}\kappa_2 \tilde{\bs d}\cdot \tilde{\bs x}}{\rm d}\tilde{\bs x}{\rm d}t\\
&=0,
\end{align*}
which means $\tilde{\bs d} \cdot \hat{\boldsymbol f}(\kappa_2 \tilde{\bs d}, \kappa_1) = 0$ for all $(\kappa_1,\kappa_2)\in H$ and $|\tilde{\bs d}| = 1.$
Since both $\tilde{\bs d}$ and $\tilde{\bs p}$ are orthonormal vectors in $\mathbb R^2$, they form an orthonormal basis in $\mathbb R^2$. Therefore we have
\[
\hat{\boldsymbol f}(\kappa_2\tilde{\bs d},\kappa_1) = \tilde{\bs d} \cdot \hat{\boldsymbol f}(\kappa_2\tilde{\bs d},\kappa_1) \tilde{\bs d} + \tilde{\bs p} \cdot \hat{\boldsymbol f}(\kappa_2\tilde{\bs d},\kappa_1) \tilde{\bs p}= 0
\]
for all $(\kappa_1,\kappa_2)\in H$ and $|\tilde{\bs d}| = 1.$ Since $\hat{\boldsymbol f}$ is analytic in $\mathbb R^3$ and $\{(\kappa_1, \kappa_2\tilde{\bs d}): (\kappa_1,\kappa_2)\in H , ~|\tilde{\bs d}|=1\}$ is an open set in $\mathbb R^3$, we have $\hat{\boldsymbol f}(\boldsymbol \xi) = 0$ for all $\boldsymbol \xi \in \mathbb R^3$, which means $\boldsymbol f(\tilde{\bs x}, t) \equiv 0$ and then
\begin{align*}
\boldsymbol J(\tilde{\bs x} - \tilde{\boldsymbol a}(t))g(t) = \boldsymbol J(\tilde{\bs x} - \tilde{\boldsymbol b}(t))g(t)
\end{align*}
for all $\tilde{\bs x} \in \mathbb R^2$ and $t>0.$ This particulary gives
\begin{align}\label{idetityp}
\boldsymbol J(\tilde{\bs x} - \tilde{\boldsymbol a}(t))= \boldsymbol J(\tilde{\bs x} - \tilde{\boldsymbol b}(t))\quad\mbox{for all}\quad t\in (0, T_0),\;\; \tilde{\bs x}=(x_1, x_2, 0).
\end{align}
Assume that there exists one time point $t_0 \in (0, T_0)$ such that $\tilde{\boldsymbol a}(t_0) \neq \tilde{\boldsymbol b}(t_0)$. By choosing $\tilde{\bs x} = \tilde{\boldsymbol a}(t_0)$ we deduce from \eqref{idetityp} that
\[
\boldsymbol J(0) = \boldsymbol J(\tilde{\boldsymbol a}(t_0) - \tilde{\boldsymbol b}(t_0)),
\]
which is a contradiction to our assumption that $\bs J\in \mathcal{A}$. This finishes the proof of $\tilde{\bs a}(t)=\tilde{\bs b}(t)$ for $t\in[0,T_0]$.
\end{proof}

\begin{rema}
The proof of Theorem \ref{up} does not depend on the Fourier transform of the
electromagnetic field in time, but it requires the data measured on the whole
surface $\Gamma_R$. However, the Fourier approach presented in the proof of
Theorems \ref{us} and \ref{ut} straightforwardly carries over to the proof of
Theorem \ref{up} without any additional difficulties. Particulary, the result of
Theorem \ref{up} remain valid with the partial data
$\{\boldsymbol{E}(\boldsymbol{x},t)\times \boldsymbol\nu: \boldsymbol
x\in \Gamma\subset \Gamma_R, t\in (0, T)\}$.
\end{rema}
\begin{rema} In the case of the scalar wave equation,
\begin{align*}
\begin{cases}
\partial^2_{t}u(\boldsymbol{x},t) +
\nabla\times(\nabla\times u
 (\boldsymbol{x},t))=  J(\tilde{\bs x}-\tilde{\boldsymbol a}(t))\,h(x_3)\,g(t),&\quad
\boldsymbol{x}\in \mathbb R^3,~ t>0,\\
u(\boldsymbol{x},0) = \partial_t u(\boldsymbol{x},0) =
0, &\quad \boldsymbol x \in \mathbb R^3,
\end{cases}
\end{align*}
where $J: \mathbb R^2\rightarrow \mathbb R_+$ is a scalar function compactly
supported on $\{(x_1,x_2)\in \mathbb R^2: x_1^2+x_2^2<\hat{R}^2 \}$. Then,
following the same arguments as in the proof of Theorem \ref{up}, one can prove
that $\tilde{\boldsymbol a}(t), t \in [0, T_0]$ can be uniquely determined by
the data set $\{\boldsymbol{E}(\boldsymbol{x},t)\times \boldsymbol\nu:
\boldsymbol x\in \Gamma\subset\Gamma_R, t\in (0, T)\}$.
\end{rema}

\section{Inverse moving source problem for a delta distribution}\label{sec:5}

As seen in the previous sections, when the temporal function $g$ is supported on
$[0, T_0]$, it is possible to recover the moving orbit function $\bs a(t)$ for
$t\in[0, T_0]$. In this section we consider the case where the temporal
function shrinks to the Dirac distribution $g(t)=\delta(t-t_0)$ with some
unknown time point $t_0>0$. Our aim is to determine $t_0$ and $\bs a(t_0)$ from
the electric data at a finite number of measurement points.

Consider the following initial value problem of the time-dependent Maxwell
equation
\begin{align}\label{Max5}
\begin{cases}
\partial^2_{t}\boldsymbol E(\boldsymbol{x},t) +
\nabla\times(\nabla\times\boldsymbol E
 (\boldsymbol{x},t))=  -\boldsymbol J(x-\boldsymbol a(t))\delta(t-t_0),&\quad
\boldsymbol{x}\in \mathbb R^3,~ t>0,\\
\boldsymbol E(\boldsymbol{x},0) = \partial_t \boldsymbol E(\boldsymbol{x},0) =
0, &\quad \boldsymbol x \in \mathbb R^3.
\end{cases}
\end{align}

Since $\nabla \cdot \boldsymbol J=0$, the electric field $\boldsymbol E(\boldsymbol x)$ in this case can be expressed as
\begin{align}\label{E}
\boldsymbol E (\boldsymbol x, t) &= \int_0^{\infty} \int_{\mathbb R^3} {\mathbb
G} ({\boldsymbol x - \boldsymbol y}, t - s) \boldsymbol J(\boldsymbol y - \boldsymbol a(s))
\delta(s-t_0){\rm d}\boldsymbol y {\rm d}s\notag\\
& = \int_0^{\infty} \int_{\mathbb R^3} \frac{1}{4\pi |\boldsymbol x -
\boldsymbol y|} \delta(|\boldsymbol x - \boldsymbol y| - (t -
s)) \boldsymbol J(\boldsymbol y - \boldsymbol a(s)) \delta(s-t_0){\rm d}\boldsymbol y {\rm d}s\notag\\
&\qquad -  \int_0^{\infty} \int_{\mathbb R^3} \nabla_{\boldsymbol x}
\nabla^\top_{\boldsymbol x} \Big( \frac{1}{4\pi |\boldsymbol x - \boldsymbol y|}
H(|\boldsymbol x - \boldsymbol y|+s-t) \Big) \boldsymbol
J(\boldsymbol y - \boldsymbol a(s)) \delta(s-t_0){\rm d}\boldsymbol y {\rm d}s\notag\\
& = \int_0^{\infty} \int_{\mathbb R^3} \frac{1}{4\pi |\boldsymbol x -
\boldsymbol y|} \delta(|\boldsymbol x - \boldsymbol y| - (t -
s)) \boldsymbol J(\boldsymbol y - \boldsymbol a(s)) \delta(s-t_0){\rm d}\boldsymbol y {\rm d}s\notag\\
&\qquad -  \int_0^{\infty} \int_{\mathbb R^3} \nabla_{\boldsymbol y}
\nabla^\top_{\boldsymbol y} \Big( \frac{1}{4\pi |\boldsymbol x - \boldsymbol y|}
H(|\boldsymbol x - \boldsymbol y|+s-t) \Big) \boldsymbol
J(\boldsymbol y - \boldsymbol a(s)) \delta(s-t_0){\rm d}\boldsymbol y {\rm d}s\notag\\
& = \int_{\mathbb R^3} \frac{1}{4\pi |\boldsymbol x - \boldsymbol y|}
\delta(|\boldsymbol x - \boldsymbol y| - (t - t_0)) \boldsymbol J(\boldsymbol y
- \boldsymbol a(t_0)) {\rm d}\boldsymbol y.
\end{align}

Before stating the main theorem of this section, we describe the strategy for
the choice of four measurement points (or receivers) on the sphere
$\Gamma_R$. The geometry is shown in Figure \ref{pg}. First, we choose
arbitrarily three different points $\boldsymbol x_1, \boldsymbol x_2,
\boldsymbol x_3 \in \Gamma_R$. Denote by $P$ the uniquely determined plane
passing through $\boldsymbol x_1$, $\boldsymbol x_2$ and $\boldsymbol x_3$, and
by $L$ the line passing through the origin and perpendicular to $P$. Obviously
the straight line $L$ has two intersection points with $\Gamma_R$. Choose one of
the intersection points with the longer distance to plane $P$ as the fourth
point $\boldsymbol x_4$. If the two intersection points have the same distance
to $P$, we can choose either one of them as $\boldsymbol x_4$. By our choice of
$\bs x_j, j=1,2,3,4$, they cannot lie on one side of any plane passing through
the origin, if the plane $P$ determined by $\bs x_j, j=1,2,3$ does not pass
through the origin.

\begin{figure}
\centering
\includegraphics[width=0.3\textwidth]{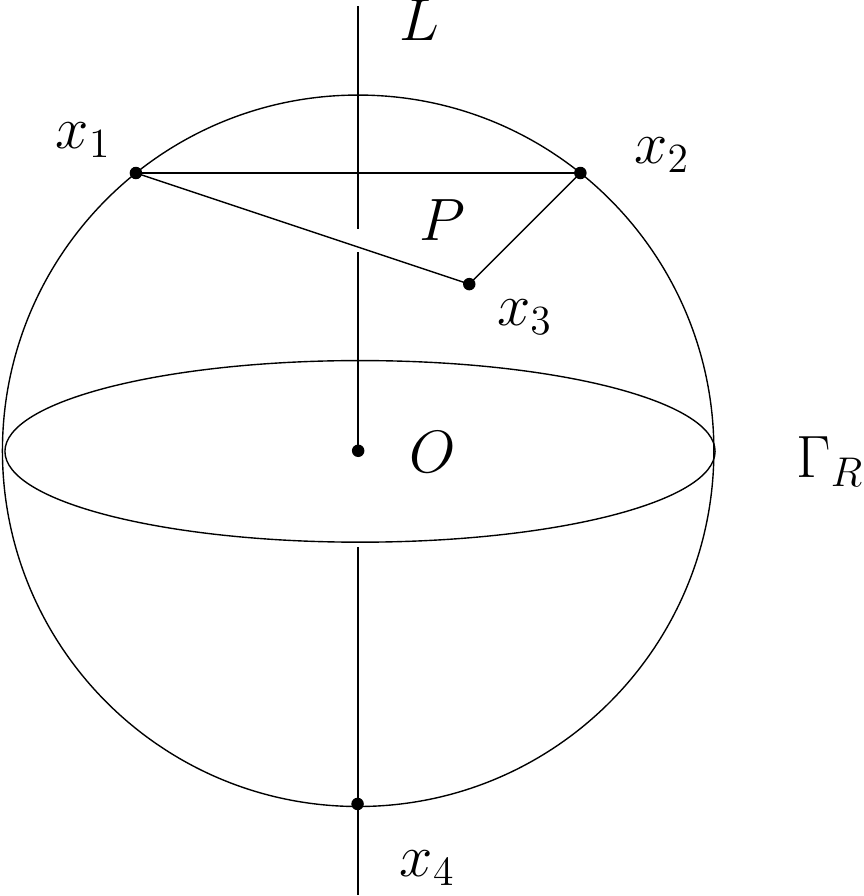}
\caption{Geometry of the four measurement points.}
\label{pg}
\end{figure}

\begin{theo}
Let the measurement positions $\boldsymbol x_j\in \Gamma_R$, $j=1,\cdots,4$ be
given as above and let $\bs J$ be specified as in the introduction part.
We assume additionally that ${\rm supp}(\boldsymbol J)= B_{\hat{R}}$ and there
exists a small constant $\delta>0$ such that $|J_i(\boldsymbol x)|>0$ for all $
\hat{R} - \delta \leq |\boldsymbol x| \leq \hat{R}$ and $i=1, 2, 3$. Then both
$t_0$ and $\boldsymbol a(t_0)$ can be uniquely determined by the data set
$\{\boldsymbol{E}(\boldsymbol x_j,t): j =1,\cdots,4,\, t\in (0, T)\}$,
where $T = t_0 + \hat{R}+R_1 + R$.
\end{theo}

\begin{proof}
Analogously to Lemma \ref{2.1}, one can prove that $\boldsymbol E (\boldsymbol
x, t) = 0$ for all $\boldsymbol x\in B_R$ and $t>T$. Taking the Fourier
transform of $\boldsymbol E(\boldsymbol x, t)$ in \eqref{fourier} with respect
to $t$ and making use of the representation of $\boldsymbol E$ in \eqref{E},
we obtain
\begin{align}\label{FF}
\hat{\boldsymbol E}(\boldsymbol x,\kappa) &= \int_{\mathbb R^3} \frac{e^{{\rm
i}\kappa(t_0 + |\boldsymbol x- \boldsymbol y|)}}{|\boldsymbol x-\boldsymbol
y|}\boldsymbol J(\boldsymbol y - \boldsymbol a(t_0)){\rm d}\boldsymbol y\notag\\
& = e^{{\rm i}\kappa t_0}\int_0^\infty e^{{\rm i}\kappa \rho} \frac{1}{\rho}
\int_{\Gamma_{\rho}(\boldsymbol x)}\boldsymbol J(\boldsymbol y - \boldsymbol
a(t_0)){\rm d}\boldsymbol y{\rm d}\rho,
\end{align}
where $\Gamma_{\rho}(\bs x):= \{\boldsymbol y\in \mathbb R^3:  |\bs
y-\boldsymbol x| = \rho\}$. Assume that there are two orbit functions
$\boldsymbol a$ and $\boldsymbol b$ and two time points $t_0$ and $\tilde{t}_0$
such that
\begin{align*}
\begin{cases}
\partial^2_{t}\boldsymbol E_1(\boldsymbol{x},t) +
\nabla\times(\nabla\times\boldsymbol E_1
 (\boldsymbol{x},t))=  -\boldsymbol J(x-\boldsymbol a(t))\delta(t-t_0),&\quad
\boldsymbol{x}\in \mathbb R^3,~ t>0,\\
\boldsymbol E_1(\boldsymbol{x},0) = \partial_t \boldsymbol E_1(\boldsymbol{x},0) =
0, &\quad \boldsymbol x \in \mathbb R^3,
\end{cases}
\end{align*}
and
\begin{align*}
\begin{cases}
\partial^2_{t}\boldsymbol E_2(\boldsymbol{x},t) +
\nabla\times(\nabla\times\boldsymbol E_2
 (\boldsymbol{x},t))=  -\boldsymbol J(x-\boldsymbol b(t))\delta(t-\tilde{t}_0),&\quad
\boldsymbol{x}\in \mathbb R^3,~ t>0,\\
\boldsymbol E_2(\boldsymbol{x},0) = \partial_t \boldsymbol E_2(\boldsymbol{x},0) =
0, &\quad \boldsymbol x \in \mathbb R^3.
\end{cases}
\end{align*}
We need to prove $t_0=\tilde{t}_0$ and $\bs a(t_0)=\bs b(\tilde{t}_0)$ under the condition
$\boldsymbol E_1(\boldsymbol x_j, t) = \boldsymbol E_2(\boldsymbol x_j, t)$ for $t \in [0, T]$ and $j=1,2,3,4$. Below we denote by $\bs x\in \Gamma_R$ one of the measurement points $\bs x_j$ ($j=1,\cdots,4$).
Introduce the functions $\bs F, \bs F_a, \bs F_b$: $\mathbb R_+\rightarrow
\mathbb R$ as follows:
\begin{align*}
\boldsymbol F(\rho) &= \frac{1}{\rho} \int_{\Gamma_{\rho}(\boldsymbol x)}\boldsymbol J(\boldsymbol y){\rm d}\boldsymbol y,\\
\boldsymbol F_a(\rho) &= \frac{1}{\rho} \int_{\Gamma_{\rho}(\boldsymbol x)}\boldsymbol J(\boldsymbol y - \boldsymbol a(t_0)){\rm d}\boldsymbol y,\\
\boldsymbol F_b(\rho) &= \frac{1}{\rho} \int_{\Gamma_{\rho}(\boldsymbol
x)}\boldsymbol J(\boldsymbol y - \boldsymbol b(\tilde{t}_0)){\rm d}\boldsymbol
y.
\end{align*}
Since $\text{supp}(\boldsymbol J)=B_{\hat{R}}$ and by our assumption, each
component  $J_j(\boldsymbol x)$ ($j=1,2,3$) is either positive or negative in a
small neighborhood of $\Gamma_{\hat{R}}$, we can obtain that
\begin{align}\label{supp}
&\inf\{\rho\in\text{supp} (\boldsymbol F)\}=|\boldsymbol x|-\hat{R},\qquad \qquad \quad \sup\{\rho\in\text{supp} (\boldsymbol F)\}= |\boldsymbol x|+\hat{R},\notag\\
&\inf\{\rho\in\text{supp} (\boldsymbol F_a)\}= |\boldsymbol x - \boldsymbol a(t_0)|-\hat{R},\quad
\sup\{\rho\in\text{supp} (\boldsymbol F_a)\}= |\boldsymbol x-\boldsymbol a(t_0)|+\hat{R},\notag\\
&\inf\{\rho\in\text{supp} (\boldsymbol F_b)\}= |\boldsymbol x-\boldsymbol
b(\tilde{t}_0)|-\hat{R},\quad \sup\{\rho\in\text{supp} (\boldsymbol
F_b)\}=|\boldsymbol x-\boldsymbol b(\tilde{t}_0)|+\hat{R}.
\end{align}
Since $\boldsymbol E_1(\boldsymbol x, t) = \boldsymbol E_2(\boldsymbol x, t), \,
t \in [0, T]$ for some point $\boldsymbol x \in \partial B_R$, from \eqref{FF}
we have
\[
e^{{\rm i}\kappa t_0} \hat{\boldsymbol F_a}(\kappa) = e^{{\rm i}\kappa
\tilde{t}_0} \hat{\boldsymbol F_b}(\kappa)
\]
for all $\kappa>0$, which means
\begin{align}\label{eq:Fourier}
\hat{\boldsymbol F_a}(\kappa) = e^{{\rm -i}\kappa (t_0 -
\tilde{t}_0)}\hat{\boldsymbol F_b}(\kappa).
\end{align}
Recalling the property of the Fourier transform,
\[
\widehat{\boldsymbol F_b(\rho- (t_0 - \tilde{t}_0))}(\kappa) = e^{{\rm -i}\kappa
(t_0 - \tilde{t}_0)} \hat{\boldsymbol F_b}(\kappa),
\]
we deduce from (\ref{eq:Fourier}) that
\[
\boldsymbol F_b(\rho - (t_0 - \tilde{t}_0)) = \boldsymbol F_a(\rho),\quad
\rho\in \mathbb R^+.
\]
Particularly,
\begin{align*}
\inf\{\text{supp} (\boldsymbol F_b(\cdot- (t_0 -
\tilde{t}_0)))\}=\inf\{\text{supp} (\boldsymbol F_a(\cdot))\},\\
\sup\{\text{supp} (\boldsymbol F_b(\cdot- (t_0 -
\tilde{t}_0)))\}=\sup\{\text{supp} (\boldsymbol F_a(\cdot))\}.
\end{align*}
Therefore, we derive from \eqref{supp} that
\begin{align*}
|\boldsymbol x-\boldsymbol b(\tilde{t}_0)| - \hat{R} + (t_0 - \tilde{t}_0)=
|\boldsymbol x-\boldsymbol a(t_0)|-\hat{R},\\
|\boldsymbol x-\boldsymbol b(\tilde{t}_0)| + \hat{R} + (t_0 - \tilde{t}_0) =
|\boldsymbol x-\boldsymbol a(t_0)|+\hat{R},
\end{align*}
which means
\begin{align}\label{identity}
|\boldsymbol x - \boldsymbol b(\tilde{t}_0)| - |\boldsymbol x - \boldsymbol
a(t_0)| =\tilde{t}_0- t_0.
\end{align}
Physically, the right and left hand sides of the above identity represent the
difference of the flight time between $\bs x$ and $\bs a(t_0)$, $\bs
b(\tilde{t}_0)$. Note that the wave speed has been normalized to one for
simplicity.

Finally, we prove that the identity (\ref{identity}) cannot hold simultaneously
for our choice of measurement points $\bs x_j\in\Gamma_R$ ($j=1,\cdots,4$).
Obviously, the set $\{\boldsymbol x \in \mathbb R^3: ~|\boldsymbol x -
\boldsymbol b(\tilde{t}_0)| - |\boldsymbol x - \boldsymbol a(t_0)| = t_0 -
\tilde{t}_0\}$ represents one sheet of a hyperboloid. This implies that $\bs
x_j$ ($j=1,2,3,4$) should be located on one half sphere of radius $R$ excluding
the corresponding equator, which is a contradiction to our choice of $\bs x_j$.
Then we have $t_0 = \tilde{t}_0$ and \eqref{identity} then becomes
\begin{align*}
|\boldsymbol x - \boldsymbol b(t_0)| - |\boldsymbol x - \boldsymbol a(t_0)| = 0.
\end{align*}
This implies that $\boldsymbol x_1, \boldsymbol x_2, \boldsymbol x_3,
\boldsymbol x_4$ should be on the same plane. This is also a contradiction to
our choice of $\boldsymbol x_i, \, i=1, \cdots, 4$. Then we have $\boldsymbol
a(t_0) = \boldsymbol b(t_0)$.
\end{proof}

\begin{rema} If the source term on the right hand side of (\ref{Max5}) takes the form
\[
\bs F(\bs x,t)=-\bs J(x-\bs a(t))\,\sum_{j=1}^m \delta(t-t_j),
\]
with the impulsive time points
\[
t_1<t_2<\cdots<t_m,\quad |t_{j+1}-t_j|>R.
\]
One can prove that the set $\{(t_j, \bs a(t_j)): j=1,2,\cdots,m\}$ can be uniquely determined by
$\{\boldsymbol{E}(\boldsymbol{x_j},t): j =1,\cdots,4,\, t\in (0, T)\}$,
where $T = t_m + \hat{R}+R_1 + R$. In fact, for $2\leq j\leq m$, one can prove that $(t_j, \bs a(t_j))$ can be uniquely determined by $\{\boldsymbol{E}(\boldsymbol{x_j},t): j =1,\cdots,4,\, t\in (T_{j-1}, T_j)\}$,
where $T_j=T_{j-1}+t_j$ and $T_1:=t_1 + \hat{R}+R_1 + R$.
\end{rema}

\section{Acknowledgement}
 The work of G. Hu is supported by the NSFC grant (No. 11671028) and NSAF grant (No. U1530401).  The work of Y. Kian is supported by  the French National
Research Agency ANR (project MultiOnde) grant ANR-17-CE40-0029.

\end{document}